\documentclass[12]{amsart}
\usepackage[all]{xy}
\usepackage{bbm}
\usepackage{amsmath}
\usepackage{amsfonts}
\usepackage{amssymb}
\usepackage{amsthm}
\usepackage{amscd}
\usepackage{latexsym}
\usepackage{txfonts}
\usepackage{multirow}
\usepackage{float}
\usepackage{threeparttable}
\usepackage{enumerate}
\usepackage{rotating}
\usepackage{tabularx}
\usepackage[colorinlistoftodos]{todonotes}
\usepackage{hyperref}
\theoremstyle{plain}
\newtheorem{theorem}{Theorem}
\newtheorem{lemma}[theorem]{Lemma}
\newtheorem{cor}[theorem]{Corollary}

\newtheorem{conj}[theorem]{Conjecture}

\newtheorem*{theorem*}{Theorem}
\theoremstyle{remark} 
\theoremstyle{remark} 
\newcommand{\bb}{\mathbb}
\newcommand{\mrm}{\mathrm}
\newcommand{\ov}{\overline}
\newcommand{\Sym}{\mathrm{Sym}}
\newcommand{\Der}{\mathfrak{Der}}
\newcommand{\I}{\mathcal{S}}

\newcommand{\ind}{\mathbf{1}}

\newcommand{\mcal}{\mathcal}
\newcommand{\mbf}{\mathbf}
\newcommand{\mfk}{\mathfrak}
\newcommand{\one}{\mathbbm{1}}
\newcommand{\Proj}{\operatorname{Proj}}
\newcommand{\de}{\mathbf{d}}
\newcommand{\pr}{\mathbf{p}}
\newcommand{\pgl}{\mathrm{PGL}}
\newcommand{\psl}{\mathrm{PSL}}

\newcommand{\agl}{\mathrm{AGL}}
\newcommand{\gl}{\mathrm{GL}}
\newcommand{\sym}{\mathrm{Sym}} 
\newcommand{\irr}{\mathrm{Irr}}
\newcommand{\ovr}{\overrightarrow}
\newcommand{\ups}{\Upsilon}
\newcommand{\vp}{V^{(2)}}

\title{All $3$-transitive groups satisfy the strict-EKR property.} 
\date{}

\author[V.R.T. Pantangi]{Venkata Raghu Tej Pantangi}
\address{Department of Mathematics and Statistics, University of Regina, Regina, Saskatchewan S4S 0A2, Canada}
\thanks{The author is funded by the Pacific Institute of Mathematical Sciences postdoctoral fellowship.}
\email{pvrt1990@gmail.com}

\begin{document}

\begin{abstract}
A subset $S$ of a transitive permutation group $G \leq \Sym(n)$ is said to be  an intersecting set if, for every $g_{1},g_{2}\in S$, there is an $i \in [n]$ such that $g_{1}(i)=g_{2}(i)$. The stabilizer of a point in $[n]$ and its cosets are intersecting sets of size $|G|/n$. Such families are referred to as canonical intersecting sets. A result by Meagher, Spiga, and Tiep states that if $G$ is a $2$-transitive group, then $|G|/n$ is the size of an intersecting set of maximum size in $G$. In some $2$-transitive groups (for instance $\Sym(n)$, $\mrm{Alt}(n)$), every intersecting set of maximum possible size is canonical. A permutation group, in which every intersecting family of maximum possible size is canonical, is said to satisfy the strict-EKR property. In this article, we investigate the structure of intersecting sets in $3$-transitive groups. 
A conjecture by Meagher and Spiga states that all $3$-transitive groups satisfy the strict-EKR property. Meagher and Spiga showed that this is true for the $3$-transitive group $\pgl(2,q)$. Using the classification of $3$-transitive groups and some results in literature, the conjecture reduces to showing that the $3$-transitive group $\agl(n,2)$ satisfies the strict-EKR property. We show that $\agl(n,2)$ satisfies the strict-EKR property and as a consequence, we prove Meagher and Spiga's conjecture. We also prove a stronger result for $\agl(n,2)$ by showing that ``large'' intersecting sets in $\agl(n,2)$ must be a subset of a canonical intersecting set. This phenomenon is called stability. 
\end{abstract}

\maketitle
\section{Introduction}
Let $n$ and $k$ be positive integers such that $n< 2k$. A family $\mcal{F}$ of $k$-subsets in $[n]$ is said to be intersecting if the intersection of any two elements of $\mcal{F}$ is non-empty. Given $i \in [n]$, the family \[\mcal{F}_{i}=\{S\subset [n] :\ |S|=k\ \&\ i \in S\},\] is a canonical example of an intersecting family of size $\binom{n-1}{k-1}$. The celebrated Erd\H{o}s--Ko--Rado(EKR) Theorem (\cite{erdos1961intersection}) states that if $\mcal{F}$ is an intersecting family, then $|\mcal{F}|\leq \binom{n-1}{k-1}$; moreover, the equality holds if and only if $\mcal{F}=\mcal{F}_{i}$, for some $i \in [n]$.   

In this article, we consider EKR-type problems for permutation groups.\\ Let $G\leq \Sym(\Omega)$ be a finite permutation group. Two permutations $g_{1},g_{2}\in G$ are said to be intersecting if and only if $g_{1}g_{2}^{-1}$ fixes a point in $\Omega$. A subset $S\subseteq G$ is said to be an \emph{intersecting set} if any pair of permutations in $S$ are intersecting. Given $\alpha,\ \beta \in \Omega$, the set $\I_{\alpha\to\beta}=\{g \in G \ :\ g(\alpha)= \beta\}$ is a canonically occurring intersecting set of size $|G|/|\Omega|$. Such intersecting sets are referred to as \emph{canonical intersecting sets}. By a {\it maximum intersecting set}, we mean an intersecting set of maximum possible size. As in the case of the classical EKR Theorem, the following two questions are of interest:

(I) What is the size of a maximum intersecting set in $G$?

(II) What is the structure of maximum intersecting sets in $G$?

Extensive research has been done to answer these questions for various infinite families of permutation groups. A permutation group is said to satisfy the \emph{EKR} property, if its canonical intersecting sets are also maximum intersecting sets. A permutation group is said to satisfy the \emph{strict-EKR} property, if every maximum intersecting set is canonical.

The groups $\Sym(n)$ (see \cite{CK2003} and \cite{DF1977}), $\mrm{Alt}(n)$ (see \cite{AM2014}), $\mrm{PGL}(2,q)$ (see \cite{MS2011}), and $\mrm{PSL}(2,q)$ (see \cite{LPSX2018}) satisfy the strict-EKR property. In \cite{MSi2021} and \cite{MST2016} the EKR properties of $2$-transitive groups were investigated. They showed that all $2$-transitive groups satisfy the EKR property and another property known as the EKR-module property. Given a transitive permutation group $G$ of degree $n$, let $U_{G}$ be the $\bb{C}$-linear span of the indicator functions of the canonical intersecting sets in $G$. The permutation group $G$ is said to satisfy the EKR-module property if $\ind_{S} \in U_{G}$ (here $\ind_{S}$ is the indicator function of $S$ in $G$) for every maximum intersecting set $S$. There are many $2$-transitive groups that do not satisfy the strict-EKR property. For instance, when $n\geq 3$, it was show in \cite{Spiga2019}, that, a maximum intersecting set in $\pgl(n,q)$, is either a coset of the stabilizer of a point or a coset of the stabilizer of a hyperplane. 

Although not all $2$-transitive groups satisfy the strict-EKR property, in \cite{MS2011}, the authors conjecture that every $3$-transitive groups satisfy the strict-EKR property. As mentioned above, the $3$-transitive groups $\mrm{Alt}(n)$, $\Sym(n)$ and $\pgl(2,q)$ satisfy the strict-EKR property. In this paper, we consider another $3$-transitive permutation group, $\agl(n,2)$, with its natural action on the $2^{n}$ points of the $n$-dimensional vector space $\bb{F}^{n}_{2}$. 

\begin{theorem}\label{thm:agln2sker}
$\agl(n,2)$ satisfies the strict-EKR property.
\end{theorem}
 
The classification of $3$-transitive groups is well known and can be found in many places in literature, for instance as Theorem 5.2 of \cite{HandbookC1}. From the classification, it follows that the socle of an almost simple $3$-transitive permutation group is either (i) $\psl(2,q)$; (ii) or $\mrm{Alt}(n)$; or (iii) one of the $3$-transitive Mathieu groups. We mentioned above that both $\psl(2,q)$ and $\mrm{Alt}(n)$ satisfy the strict-EKR property. In \cite{AM2015}, the authors checked that each of the $3$-transitive Mathieu groups satisfies the strict-EKR property. Now, from Theorem~7.2 of \cite{AM2014}, it follows that every almost simple $3$-transitive group satisfies the strict-EKR property. From Theorem 5.2 in \cite{HandbookC1}, it follows that an affine $3$-transitive group is either isomorphic to $\agl(n,2)$ or is isomorphic to $\mrm{Alt}(7) \ltimes \bb{F}^{4}_{2} $. It can be checked with the help of a computer that $\mrm{Alt}(7) \ltimes \bb{F}^{4}_{2}$ satisfies the strict-EKR property.
So $\agl(n,2)$ is the only $3$-transitive group whose EKR properties were previously unknown. So, as a consequence of Theorem~\ref{thm:agln2sker}, we complete the proof of Conjecture~3 of \cite{MS2011}.

\begin{cor}\label{cor:3t}
Every $3$-transitive group satisfies the strict-EKR property.
\end{cor}

In \cite{HM67}, Hilton and Milner gave the following stronger version of the classical EKR theorem.

\begin{theorem}[Hilton--Milner Theorem \cite{HM67}]
Suppose that $n>2k\geq 4$ and $\mcal{F}\subset \binom{[n]}{k}$ is an intersecting family with $|\mcal{F}| > \binom{n-1}{k-1}- \binom{n-k-1}{k-1}+1$, then $\mcal{F}$ must be a subset of a canonical intersecting family.     
\end{theorem}
In  \cite{ellis2012proof}, Ellis proved a Hilton--Milner analogue for intersecting sets in $\Sym(n)$, which was conjectured by Cameron--Ku in \cite{CK2003}. An alternative proof was given in \cite{ellis2015quasi}.   
\begin{theorem}(Corollary~3.5 of \cite{ellis2012proof})\label{thm:snstablity}
There exists a constant $c_{0}<1$ such that any intersecting set $S\subseteq \Sym(n)$ of size at least $c_{0}(n-1)!$ must be a subset of a canonical intersecting set in $\Sym(n)$.
\end{theorem}

Using the methods in \cite{ellis2015quasi}, in \cite{plaza2015stability}, Plaza extended the above result to $\mrm{PGL}(2,q)$ (see \cite[Theorem~2]{plaza2015stability}).

We were able to extend these methods to get the following analogous results for $\agl(n,2)$.
\begin{theorem}\label{thm:agln2stablity}
There exists a constant $c_{0}<1$ such that any intersecting set $S\subseteq \agl(n,2)$ of size at least $c_{0}|\gl(n,2)|$ must be a subset of a canonical intersecting set in $\agl(n,2)$.
\end{theorem}  

 We now describe the structure of the paper. We start with \S~\ref{sec:notationandbackground}, in which we establish some notation and mention a few useful general results. Once this is done, we prove Theorem~\ref{thm:agln2sker} in \S~\ref{sec:strict}. We prove it by assuming the veracity of a technical result, Lemma~\ref{lem:charactersums}. We prove this technical result in \S~\ref{sec:charactersums}. The proof of Theorem~\ref{thm:agln2stablity} is done in \S~\ref{sec:stability}. 

\section{Notation and Background.}\label{sec:notationandbackground}

In this section, we establish some notation and collect some useful results from character theory and algebraic graph theory. We first start by describing the group $\agl(n,q)$.

\subsection{The group $\agl(n,q)$}
Let $q$ be a power of a prime, and let $\bb{F}_{q}$ denote the field with $q$ elements. Let $V$ denote the $n$-dimensional vector space $\bb{F}^{n}_{q}$ of $n$-columns with entries in $\bb{F}_{q}$. By $(\mbf{e}_{1},\ \ldots,\ \mbf{e}_{n})$, we denote the standard ordered basis of $V$.
By the group $\agl(n,2)$, we mean the semidirect product $\gl(n,q) \ltimes V$. Elements of $\agl(n,2)$ are ordered pairs of the form $(M,\ v)$ with $M \in \gl(n,q)$ and $v \in V$. Given $(M_{1},\ v_{1})$ and $(M_{2},\ v_{2})$ in $\agl(n,q)$, we have $(M_{1},v_{1})\cdot (M_{2},\ v_{2})=(M_{1}M_{2}, v_{1}+M_{1}v_{2})$. The natural action of $\agl(n,q)$ on $V$ satisfies 
\[(M,\ v)\cdot w =v+Mw,\] for all $(M,\ v) \in \agl(n,q)$ and $w \in V$. This action is $2$-transitive for all $q$ and $3$-transitive when $q=2$. 

\subsection{The Module Method.}\label{sec:mm}
We now describe the Module Method: a sufficient condition for a $2$-transitive group to satisfy the strict-EKR property. This was discovered by Ahmedi and Meagher in \cite{AM2015}. In the same paper, they use this result to prove strict-EKR property of some $2$-transitive groups. This module method was also used in \cite{MS2011} and \cite{LPSX2018}.
Let $G \leq \sym(n)$ be a $2$-transitive permutation group. By a \emph{derangement} in $G$, we mean a fixed-point-free permutation in $G$. Let $\Der(G)$ be the set of derangements in $G$. By $[n]^{(2)}$, we denote the set of ordered pairs of distinct elements in $[n]:=\{1,2,\ldots, n\}$. We define the \emph{derangement matrix} $M(G)$ of $G$ to be the matrix indexed by $\Der(G) \times [n]^{(2)}$ with \[M(G)_{d,\ (a,b))} =\begin{cases} 1 & \text{if $d(a)=b$ and}\\
0 & \text{otherwise.}
\end{cases}
\]  
The following result--known as the Module Method--is essentially Theorem~4.5 of \cite{AM2015}. The main result of \cite{MSi2021} and \cite{MST2016} show that conditions (a) and (b) of Theorem~4.5 of \cite{AM2015} are true for all $2$-transitive groups. Taking these results into account, we take the liberty of restating the Module Method as follows:

\begin{theorem}\label{thm:modulemethod}
 Let $G \leq \sym(n)$ be a $2$-transitive permutation group and let $M(G)$ be its derangement matrix. If $\mrm{rank}(M(G))=(n-1)(n-2)$, then $G$ satisfies the strict-EKR property.   
\end{theorem}

Given a group $G$, by $\bb{C}[G]$, we denote its group algebra. Let $\Omega$ be a $G$-set. Let $\bb{C}[\Omega]$ be the $|\Omega|$-dimensional $\bb{C}$-vector space generated by $\{\ups_{\omega} :\ \omega \in \Omega\}$ as a basis. The action of $G$ on $\Omega$ leads to a realization of $\bb{C}[\Omega]$ as a $G$-module. Given $S \subseteq \Omega$, we use $\ups_{S}$ to denote $\sum\limits_{\omega \in S} \ups_{\omega}$. The group $G$ acts on both $[n]^{(2)}$ and on $\Der(G)$ (via conjugation). The matrix $M(G)$ can be considered as a matrix representation of some  $\mfk{M} \in \mrm{Hom}_{\bb{C}[G]}\left(\bb{C}[[n]^{(2)}] ,\ \bb{C}[\Der(G)] \right)$. If $U$ is an irreducible submodule of $\bb{C}[[n]^{(2)}]$, then by Schur's lemma, either $\mfk{M}(U) \cong U$ or $\mfk{M}(U)\cong \{0\}$. So, characterizing irreducible submodules which are ``annihilated'' by $M(G)$ will give us the rank of $M(G)$. We prove Theorem~\ref{thm:agln2sker} using the module method.     
\subsection{Delsarte--Hoffman ratio bound and stability.}\label{sec:stabilitynotation}
The investigation of intersecting sets in a permutation group can be transferred to that of independent sets in a certain Cayley graph associated with the permutation group, called the derangement graph. Given a permutation group $G \leq \sym(n)$ with $\Der(G)$ as the set of derangements in $G$, the derangement graph $\Gamma_{G}$ is the Cayley graph on $G$ with $\Der(G)$ as its ``connection'' set.
It is elementary to observe that a set $\I \subset G$ is intersecting if and only if $S$ is an independent set in $\Gamma_{G}$. This identification of intersecting sets with independent sets empowers us to use spectral graph theoretic results. This methodology has proved useful in characterization of intersecting sets in many groups, for instance, see \cite{CK2003}, \cite{GM2009}, and \cite{MS2011}.   

We now describe a well-known spectral bound on independence number of regular graphs. We first establish some notation. Given a graph $\Gamma=(V,E)$, by $\bb{C}^{V}$, we denote the set of $\bb{C}$-valued functions on the vertex set $V$. The adjacency matrix of $\Gamma$, can be considered as a linear map in $\mrm{End}_{\bb{C}}\left(\bb{C}^{V} \right)$. By an eigenvalue of $\Gamma$, we mean an eigenvalue of its adjacency matrix. Given an eigenvalue $\zeta$ of $\Gamma$, by $V_{\zeta}$, we denote the $\zeta$-eigenspace in $\bb{C}^{V}$. Given a subset $S\subseteq V$, by $\ind_{S}$, we denote the indicator function of $S$ as a subset of $V$. The following result--known as the Delsarte--Hoffman ratio bound--gives an upper bound on the independence number of a regular graph. The statement we provide is \cite[Theorem~2.4.2]{GMbook}.
\begin{theorem}\label{thm:ratiobound}
 Let $\Gamma=(V, E)$ be a $k$ regular graph on $v$ vertices. If $\lambda$ is the least eigenvalue of $\Gamma$, then for any independent set $S$, we have 
 \[|S| \leq \dfrac{v}{1- \dfrac{k}{\lambda}},\] and if equality holds, then 
 \[\ind_{S} \in V_{k}+V_{\lambda}.\]
\end{theorem}

The following is a useful corollary of the ratio bound.
\begin{cor}[Corollary 2.4.3 of \cite{GMbook}]\label{cor:ratiobound}
 Let $\Gamma=(V, E)$ be a $k$ regular graph on $v$ vertices with  $\lambda$ as its least eigenvalue. If $S$ is an independent set in $X$ such that $|S|= \dfrac{v}{1-\dfrac{k}{\lambda}}$, then each vertex not in $S$ has exactly $-\lambda$ neighbours in $S$.  
\end{cor}
The vector space $\bb{C}^{V}$ comes with a natural inner product whose norm $||\cdot||$ satisfies \[||f||^{2}= \frac{\sum\limits_{x \in V}f(x)\ov{f(x)}}{|V|},\] for all $f \in \bb{C}^{V}$.  Ellis (see \cite[Lemma 3.2]{ellis2012proof}) provided the following generalization of the ratio bound.
\begin{theorem}[Ellis]\label{thm:stabilityratiobound}
 Let $\Gamma=(V, E)$ be a $k$ regular graph on $v$ vertices.  Let $k=\lambda_{1}\geq \lambda_{2}\ldots \geq \lambda_{v}$ be its eigenvalues and let $U:=\left\langle \ind_{V} \right\rangle \oplus V_{\lambda_{v}}$.  Let $\mu=\mrm{min}\{\lambda_{i}\ :\ \lambda_{i} \neq \lambda_{k}\}$. If $S$ is an independent set in $\Gamma$ with $|S|/v =c$, then 

 \[||\ind_{S}-\Proj_{U}(\ind_{S})||^2 \leq \frac{c|\lambda_{v}|-c^{2}(k-\lambda_{v})}{|\lambda_{v}|-|\mu|}.\]
 
\end{theorem}

Roughly speaking, the above result shows that if $S$ is a ``large'' independent set in $\Gamma$, then $\ind_{S}$ is ``close'' (in Euclidean sense) to the space $U$. We use this to prove Theorem~\ref{thm:agln2stablity}. We finish this section, by mentioning a result that helps us compute the eigenvalues of derangement graphs.

We now recall some results on eigenvalues and eigenspaces of $\Gamma_{G}$. Spectra of adjacency matrix of derangement graphs follows from results by Babai (\cite{Babai}) and Diaconis--Shahshahani (\cite{DS1981}) on normal Cayley graphs. Before describing the result, we establish some notation. The vector space $\bb{C}^{G}$ can be considered as the group algebra of $G$. Given $\eta \in \irr(G)$, let $E_{\eta} \in \bb{C}^{G}$ be the map satisfying 
\[E_{\eta}(g)= \dfrac{\eta(1)}{|G|} \sum\limits_{g \in G} \eta(g^{-1})g.\] By $U_{\eta}$, we denote the $2$-sided ideal of $\bb{C}^{G}$ generated by $E_{\eta}$. It is well known (see for instance \cite{isaacs2008character}) that $U_{\eta}$ is a $\eta(1)^{2}$-dimensional subspace. 
 The following lemma, which is a special case of Lemma~5 of \cite{DS1981}, describes the spectrum of the adjacency matrix of $G$. 

\begin{lemma}{\rm(Babai, Diaconis--Shahshahani)}\label{lem:spec}
Let $G$ be a permutation group of degree $n$, with $\Der(G) \subset G$ being the set of derangements.  Then the set of eigenvalues of $\Gamma_{G}=\mrm{Cay}(G,\ \Der(G))$ is 
$\{\lambda_{\eta}\ :\ \eta \in \mrm{Irr}(G)\}$, where $$\lambda_{\eta}=\frac{1}{\eta(1)}\sum\limits_{g \in D(G)}\eta(g).$$ Given an eigenvalue $\nu$, then the $\nu$-eigenspace in $\bb{C}^{G}$ is the two-sided ideal
$$\sum\limits_{\{\eta \ :\ \eta\in \mrm{Irr}(G)\ \text{and}\ \lambda_{\eta}=\nu\}} U_{\eta}.$$
\end{lemma}

\subsection{Character Theory of $\agl(n,2)$}\label{sec:characters}
In this section, we describe the irreducible characters of $G:=\agl(n,2)$ and some permutation characters arising from some of its natural actions. By $\one$, we denote the trivial character of $G$.

The irreducible characters of $G$ can be constructed by using the characters of $V:=\bb{F}^{n}_{2}$ and $G_{0}:=\gl(n,2)$. We now describe this construction. The group $G_{0}$ acts on $\irr(V)$ via 
\begin{center}
$M\cdot \chi(v)=\chi(Mv)$, for all $M \in G_{0}$, $\chi \in \irr{V}$, and $v\in V$.   
\end{center}
The action of $G_{0}$ on $V$ has two orbits: $\{\mathbf{0}\}$ and $V\setminus \{\mathbf{0}\}$. Therefore, the action of $G_{0}$ on $\irr(V)$ has two orbits: $\one_{V}$ and $\irr(V)\setminus \{\one_{V}\}$. We now describe two types of representations of $G$:
\begin{enumerate}
\item Given an irreducible representation $\Phi$ of $G_{0}$, by $\tilde{\Phi}$, we denote the irreducible representation of $G$ obtained by composing $G$ with the canonical projection from $G$ to $G_{0}$.
\item Let $\delta$ be a fixed non-trivial linear character of $V$. The stabilizer $G_{0,\delta}$ of $\delta$ in $G_{0}$ is isomorphic to $\agl(n-1,2)$. The linear character $\delta$ extends to a linear character of $G_{0,\delta}V$ by setting $\delta(N,v)= \delta(v)$, for all $N \in G_{0,\delta}$ and $v \in V$. Given an irreducible representation $\rho$ of  $G_{0,\delta}$, let $\ov{\rho}$ be the irreducible representation of $G_{0,\delta}V$ obtained by composing $\rho$ with the canonical projection from $G_{0,\delta}V$ to $K$. The representation $\mrm{Ind}^{G}_{G_{0,\delta}V} (\delta \otimes \ov{\rho})$ is an irreducible representation of $G$. 
\end{enumerate}

Using the well-known characterization of complex irreducible representations of semidirect products by abelian groups (see for instance Proposition 25 of \S~8.2 of \cite{serre1977linear}), every representation of $G$ is one of the above two types. 

\begin{lemma}\label{lem:affineirr}
Let $\delta$ be an arbitrary but fixed non-trivial irreducible character of $V$. Then 
\[\irr(G)=\{\tilde{\phi}\ :\ \phi \in \irr(G_{0})\} \cup \{\mrm{Ind}^{G}_{G_{0,\delta}V} (\delta \otimes \ov{\rho})\ :\ \rho \in \irr(G_{0,\delta}) \}.\]
\end{lemma}

We now describe permutation characters of $G$ arising from some natural actions.  Let $\pi$ be the permutation character associated with the action of $G$ on $V$. By $V^{\{2\}}$, denote the set of $2$-subsets of $V$, and let $V^{(2)}$ denote the set of ordered pairs of distinct elements in $V$. Let $\pi^{\{2\}}$ and $\pi^{(2)}$ be the permutation characters associated with the $G$-sets $V^{\{2\}}$ and $V^{(2)}$, respectively. Let $H$ denote the stabilizer in $G$ of $(\mbf{0},\ \mbf{e}_{n}) \in V^{(2)}$ and let $K$ be the stabilizer of $\{\mbf{0},\ \mbf{e}_{n}\} \in V^{\{2\}}$. We now note that $\pi^{(2)}= \mrm{Ind}^{G}_{H}(\one)$ and $\pi^{\{2\}}= \mrm{Ind}^{G}_{K}(\one)$. We now describe the decomposition of $\pi^{(2)}$ and $\pi^{\{2\}}$ as sums of irreducible characters.

The action of $G$ on $V$ is $2$-transitive, and so $\psi:=\pi-\one$ is an irreducible character. The action of $H$ on $V$ has $3$ orbits and the action of $K$ on $V$ has $2$ orbits. Using Frobenius reciprocity and Burnside's orbit counting lemma, we have
\begin{align*}
    3&= \dfrac{1}{|H|} \sum\limits_{h \in H} (1+\psi(h)) = 1 + \left\langle\psi,\ \mrm{Ind}^{G}_{H}(\one) \right\rangle \\
    2&= \dfrac{1}{|K|} \sum\limits_{k \in K} (1+\psi(k)) = 1 + \left\langle\psi,\ \mrm{Ind}^{G}_{K}(\one) \right\rangle.
\end{align*}

Thus, 
\begin{equation}\label{eq:stdmultiplicity}
\left\langle \psi,\ \pi^{(2)} \right\rangle =2\ \&\ \ \left\langle \psi,\ \pi^{\{2\}} \right\rangle =1.
\end{equation}

Next, we consider the $2$-transitive action of $\gl(n,2)$ on $V \setminus \{\mbf{0}\}$. As $\gl(n,2)$ is a quotient of $G$ by $V$, we can extend this to a $2$-transitive action of $G$ on $V \setminus \{\mbf{0}\}$. Let $\theta \in \irr(G)$ be such that $\one + \theta$ is the permutation character with respect to this action. We note that $\mrm{ker}(\theta)=V$. We note that the actions of both $H$ and $K$  on $V\setminus\{0\}$ have exactly $2$ orbits. Now, arguing as above, we can conclude that 

\begin{equation}\label{eq:glstdamultiplicity}
\left\langle \theta,\ \pi^{(2)} \right\rangle =1\ \&\ \left\langle \theta,\ \pi^{\{2\}} \right\rangle =1.
\end{equation}

We observe that the orbits of $K$ on $V^{\{2\}}$ are
\begin{subequations}\label{eq:Korbitssets}
\begin{align}
\mfk{O}_{1}^{\{2\}} :=& \{\{\mbf{0},\ \mbf{e}_{n}\}\};\\
\mfk{O}_{2}^{\{2\}} :=&\{\{\mbf{0},\ v\} :\ v \in V\setminus \{\mbf{0},\ \mbf{e}_{n}\}\} \cup \{\{\mbf{e}_{n},\ v\} :\ v \in V\setminus \{\mbf{0},\ \mbf{e}_{n}\}\};\\
\mfk{O}_{3}^{\{2\}} :=&\{\{v,\ w\} :\ v, w \in V\setminus \{\mbf{0},\ \mbf{e}_{n}\}\ \&\ v+w=\mbf{e}_{n}\};\ \text{and}\\
\mfk{O}_{4}^{\{2\}} :=&\{\{v,\ w\} :\ v,\ w,\ v+w \in V\setminus \{\mbf{0},\ \mbf{e}_{n}\}\}.
\end{align}
\end{subequations}

As $K$ as $4$ orbits on $V^{\{2\}}$ and $\pi^{\{2\}}=\mrm{Ind}^{G}_{K}(\one)$, we can conclude that $\left\langle \pi^{\{2\}},\ \pi^{\{2\}} \right\rangle =4$, and thus that

\begin{equation}\label{eq:alphadef}
\alpha := \pi^{\{2\}}-\one-\theta -\psi
\end{equation}  is an irreducible character. 

We observe that the orbits of $H$ on $V^{\{2\}}$ are precisely $\mfk{O}_{1}^{\{2\}}$, $\mfk{O}_{3}^{\{2\}}$, $\mfk{O}_{3}^{\{2\}}$,\\ $\{\{\mbf{0},\ v\} :\ v \in V\setminus \{\mbf{0},\ \mbf{e}_{n}\}\}$ and  $\{\{\mbf{e}_{n},\ v\} :\ v \in V\setminus \{\mbf{0},\ \mbf{e}_{n}\}\}$. Thus, arguing as above by using orbit-counting and Frobenius reciprocity, we have $\left\langle \pi^{\{2\}},\ \pi^{(2)} \right\rangle =5$. Now using \eqref{eq:stdmultiplicity}, \eqref{eq:glstdamultiplicity}, and \eqref{eq:alphadef}, we have 
\begin{equation}\label{eq:alphamultiplicity}
\left\langle \pi^{(2)},\ \alpha \right\rangle =1.
\end{equation}
So far, we described four distinct irreducible summands of $\pi^{(2)}$. We claim that 
\begin{equation}\label{eq:betadef}
\beta :=\pi^{(2)}-2\psi-\theta-\alpha
\end{equation}
is irreducible. To show this, we first observe that the action of $H$ on $V^{(2)}$ (set of distinct pairs of points in $V$) has $8$ orbits, namely,

\begin{subequations}\label{eq:Horbitspairs}
\begin{align}
\mfk{Q}_{1}^{(2)} :=& \{ (\mbf{0},\ \mbf{e}_{n})\};\\
\mfk{Q}_{2}^{(2)} :=& \{ (\mbf{e}_{n},\ \mbf{0} )\};\\
\mfk{Q}_{3}^{(2)} :=&\{(\mbf{0},\ v) :\ v \in V\setminus \{\mbf{0},\ \mbf{e}_{n}\}\};\\
\mfk{Q}_{4}^{(2)} :=&\{(v,\ \mbf{0}) :\ v \in V\setminus \{\mbf{0},\ \mbf{e}_{n}\}\};\\
\mfk{Q}_{5}^{(2)} :=&\{(\mbf{e}_{n},\ v) :\ v \in V\setminus \{\mbf{0},\ \mbf{e}_{n}\}\};\\
\mfk{Q}_{6}^{(2)} :=&\{(v,\ \mbf{e}_{n}) :\ v \in V\setminus \{\mbf{0},\ \mbf{e}_{n}\}\};\\
\mfk{Q}_{7}^{(2)} :=&\{(v,\ w) :\ v, w \in V\setminus \{\mbf{0},\ \mbf{e}_{n}\}\ \&\ v+w=\mbf{e}_{n}\};\ \text{and}\\
\mfk{Q}_{8}^{(2)} :=&\{(v,\ w) :\ v,\ w,\ v+w \in V\setminus \{\mbf{0},\ \mbf{e}_{n}\}\}.
\end{align}
\end{subequations}

Therefore, we have $\left\langle \pi^{(2)},\ \pi^{(2)}  \right\rangle=8$, and thus using \eqref{eq:stdmultiplicity},\ \eqref{eq:glstdamultiplicity}, \eqref{eq:alphamultiplicity}, we can conclude that $\beta$ is an irreducible character with $\left\langle \pi^{(2)},\ \beta  \right\rangle=1$. We have proved the following

\begin{lemma}\label{lem:permcharsetdecom}
\begin{enumerate}
\item  $\pi^{\{2\}}=1+\psi +\theta +\alpha $ is the decomposition of $\pi^{\{2\}}=\mrm{Ind}^{G}_{K}(\one)$ into  irreducible characters of $G$.
    \item $\pi^{(2)}=1+2\psi +\theta +\alpha +\beta$ is the decomposition of $\pi^{(2)}=\mrm{Ind}^{G}_{H}(\one)$ into  irreducible characters of $G$.
\end{enumerate}
\end{lemma}
\section{Proof of Theorem~\ref{thm:agln2sker}.}\label{sec:strict}
We recall that the group $G:=\agl(n,2)$ acts $3$-transitively on the vector space $V := \bb{F}_{2}^{n}$. Let $V^{(2)}$ denote the set of ordered pairs of distinct elements in $V$. The derangement matrix $M(G)$ is the matrix indexed by $\Der(G) \times V^{(2)}$, satisfying \[M(G)_{d,\ (a,b))} =\begin{cases} 1 & \text{if $d(a)=b$,}\\
0 & \text{otherwise.}
\end{cases}
\]  
Using the module method (Theorem~\ref{thm:modulemethod}), to prove Theorem~\ref{thm:agln2sker}, it suffices to show that the rank of the derangement matrix $M(G)$ is $(2^{n}-1)(2^{n}-2)$.  In this section, we prove this equality by assuming the following technical result, which we prove in \S~\ref{sec:charactersums}. 

\begin{lemma}\label{lem:charactersums}
\begin{enumerate}
    \item For $\eta \in \{\one,\ \theta,\ \alpha,\ \beta \}$, we have  \[\sum\limits_{s \in S} \eta(s^{-1})\neq 0.\]
    \item For all $t \in \mfk{C}$, \[\sum\limits_{\substack{g \in G\\ tg(\mbf{0})= g(\mbf{e}_{n})}} \eta(g^{-1}) =0,\]    
\end{enumerate}
\end{lemma}

Let $M$ to be the linear map from $\bb{C}[V^{(2)}]$ to $\bb{C}[\Der(G)]$ satisfying \[M(\ups_{(a,b)}):= \sum\limits_{\substack{g \in \Der(G)\\ g(a)=b}} \ups_{g}.\] The matrix representation of $M$ with respect to standard bases is $M(G)$. Following an argument in \S~3.2 of \cite{LPSX2018}, we first show that $\mrm{Rank}(M(G)) \leq (2^{n}-1)(2^{n}-2)$. Given $(a,b) \in \vp$, we consider the following vectors in $\bb{C}[\vp]$,
\begin{subequations}\label{eq:kernelvectors}
  \begin{align}
      l_{(a,b)}& := \sum\limits_{\substack{v \in V\\ v\neq a,b}}\left( \ups_{(a,v)}-\ups_{(b,v)} \right) + \ups_{(a,b)}- \ups_{(b,a)} \\
      r_{(a,b)}& := \sum\limits_{\substack{v \in V\\ v\neq a,b} }\left(\ups_{(v,a)}-\ups_{(v,b)}\right) + \ups_{(b,a)}- \ups_{(a,b)}
  \end{align}  
\end{subequations}

We define the subspaces 
\begin{subequations}\label{eq:kernelspaces}
\begin{align}
V^{(2)}_{l}:&=\mrm{Span}\left(\{l_{(a,b)}\ :\ (a,b) \in V^{(2)}\} \right)\\
V^{(2)}_{r}:&=\mrm{Span}\left(\{r_{(a,b)}\ :\ (a,b) \in V^{(2)}\} \right)
\end{align}
\end{subequations}

Using essentially the same argument as in the proof of Lemma~18 of \cite{LPSX2018}, we have the following:
\begin{enumerate}
    \item $\mrm{dim}(V^{(2)}_{l})= 2^{n}-1=\mrm{dim}(V^{(2)}_{r})$ and
    \item $\vp_{l} \cap \vp_{r}= \{0\}$.
\end{enumerate}
Observing that $M(l_{(a,b)})=0=M(r_{(a,b)})$, we observe that \[\mrm{dim}(\mrm{Ker}(M)) \geq \mrm{dim}(V^{(2)}_{l}) + \mrm{dim}(V^{(2)}_{r})\geq 2(2^{n}-1).\] Now by the rank-nullity theorem, we have 
\begin{equation}\label{eq:rankup}
 \mrm{Rank}(M(G)) =\mrm{Rank}(M) \leq (2^{n}-1)(2^{n}-2).   
\end{equation}
Let $J_{n}\in \gl(n,2)$ denote the Jordan matrix of dimension $n$ and $1$ as its eigenvalue, that is,
\[J_{n}= \begin{pmatrix}
1 & 1 & 0 & \cdots & 0& 0\\
0 & 1 & 1 & \cdots & 0 &0 \\
\vdots & \vdots & \vdots & \ddots &\vdots & \vdots\\
0 & 0 &0 & \cdots & 1 & 1 \\
0 & 0 & 0 & \cdots &0 & 1
\end{pmatrix}.\]

Let $\mbf{e}_{n}$ denote the $n$-th element of the elementary basis of $V$. We consider the element $\mbf{c}:=(J_{n},\ \mbf{e}_{n}) \in G$. For any $v \in V$, we have $\mbf{c}\cdot v = \mbf{e}_{n} + J_{n}v$. Since $\mbf{e}_{n} \notin \mrm{Im}(J_{n}-I)$, we must have $\mbf{c}\cdot v \neq v$, for all $v \in V$, that is, $\mbf{c} \in \Der(G)$. Let $\mfk{C}$ denote the set of all conjugates of $\mbf{c}$ in $G$. By $\mcal{M}$, we denote the submatrix of $M(G)$ whose rows are indexed by $\mfk{C}$ and whose columns by $V^{(2)}$.

Let $\mfk{m}$ be the linear map $\mfk{m} : \bb{C}[V^{(2)}] \to \bb{C}[\mfk{C}]$ satisfying \[\mfk{m}(\ups_{(a,b)})= \sum\limits_{t \in \mfk{C}} \mcal{M}_{t,\ (a,b)}\ups_{t} =\sum\limits_{\substack{t \in \mfk{C} \\ t(a)=b}} \ups_{t}\] Clearly, 
\begin{equation}\label{eq:rankcomprision}
\mrm{rank}(\mfk{m})= \mrm{rank}(\mcal{M}) \leq \mrm{rank}(M(G)).
\end{equation}
As $xtx^{-1}(x(a))=x(b)$, it follows that $\mfk{m}$ is in fact a $G$-module homomorphism.

 We will now recall some basic facts on group algebra, proofs of which can be found in any standard text on representation theory, such as \cite{isaacs2008character}. The group algebra is a semisimple algebra, and so by the Artin--Wedderburn theorem, it decomposes as a direct sum of simple $2$-sided ideals. These ideals can be described in terms of irreducible characters of $G$. By $\irr(G)$, we denote the set of irreducible complex characters of $G$. We denote the trivial character of $G$ by $\one$. Associated with every $\eta \in \irr(G)$, there is a central primitive idempotent $e_{\eta} := \frac{\eta(1)}{|G|} \sum\limits_{g \in G} \eta(g^{-1}) g$. . The group algebra $\bb{C}G$ decomposes into simple ideals, in the following manner:
\begin{equation}\label{eq:groupalgebrade}
\bb{C}G = \bigoplus\limits_{\eta \in \irr(G)} \bb{C}G\cdot e_{\eta}.
\end{equation} 

In the above decomposition, $\bb{C}G\cdot e_{\eta}$ is the isotypic component of $\bb{C}G$ corresponding to $\eta$. By $2$-transitivity of the action of $G$ on $V$, the permutation module $\bb{C}[V^{(2)}]$ is generated (as a $\bb{C}[G]$-module) by the element $\ups_{\mbf{0},\ \mbf{e}_{n}}$.
 Now, using \eqref{eq:groupalgebrade}, we have 
\begin{align*}
\bb{C}[V^{(2)}] & = \bigoplus\limits_{\eta \in \irr(G)} \bb{C}G\cdot e_{\eta} \ups_{\mbf{0},\ \mbf{e}_{n}} 
\end{align*} 

By the definition of $e_{\eta}$ and $2$-transitivity of $G$, we have 
\begin{align*}
e_{\eta} \ups_{\mbf{0},\ \mbf{e}_{n}} & = \frac{\eta(1)}{|G|} \sum\limits_{g \in G} \eta(g^{-1}) \ups_{g(\mbf{0}),\ g(\mbf{e}_{n})}\\
& = \frac{\eta(1)}{|G|} \sum\limits_{(a,b)\in \Omega}\left(\sum\limits_{\substack{g \in G\\ g(\mbf{0},\ \mbf{e}_{n})=(a,b)}} \eta(g^{-1}) \right) \ups_{a,b}.
\end{align*}

Given $\eta \in \irr(G)$, the $\eta$-isotypic component in $\bb{C}[V^{(2)}]$ is 
\begin{equation}\label{eq:isotypiccomponent}
\vp_{\eta}:= \bb{C}[G]\cdot e_{\eta} \ups_{\mbf{0},\ \mbf{e}_{n}}. 
\end{equation}
Using Lemma~\ref{lem:permcharsetdecom}, we have

\begin{align*}
\bb{C}[\vp] & =  \vp_{\one} \bigoplus \vp_{\psi}  \bigoplus \vp_{\theta} \bigoplus \vp_{\alpha} \bigoplus \vp_{\beta}.
\end{align*}

Using the above, we have 
\begin{align}\label{eq:imagem}
\mfk{m}\left(\bb{C}[\vp]\right) & = \mfk{m} \left( \vp_{\one} \right) \bigoplus \mfk{m} \left( \vp_{\psi} \right)  \bigoplus \mfk{m} \left(\vp_{\theta} \right) \bigoplus \mfk{m}\left( \vp_{\alpha} \right)\bigoplus \mfk{m} \left( \vp_{\beta} \right).
\end{align}

By  Lemma~\ref{lem:permcharsetdecom}, for all $\eta \in \{\one,\ \theta,\ \alpha,\ \beta \}$,  $\vp_{\eta}$ is an irreducible module of dimension $\eta(1)$.  Since $\mfk{m}$ is a $G$-module map, by Schur's lemma, for all 
$\eta \in \{\one,\ \theta,\ \alpha,\ \beta \}$, either $\mfk{m}(\vp_{\eta}) \cong \vp_{\eta}$ or $\mfk{m}(\vp_{\eta}) \cong \{0\}$. By \eqref{eq:isotypiccomponent}, since  $e_{\eta} \ups_{\mbf{0},\ \mbf{e}_{n}}$ generates $V_{\eta}$, we have $\mfk{m}(\vp_{\eta}) \neq 0$ if and only if $\mfk{m}\left(e_{\eta} \ups_{\mbf{0},\ \mbf{e}_{n}}  \right)\neq 0$.

As $\mfk{m}$ is a $G$-module map, we have 
\begin{align}\label{eq:imageiso}
\mfk{m}\left(e_{\eta} \ups_{\mbf{0},\ \mbf{e}_{n}} \right) & = e_{\eta}\mfk{m}\left( \ups_{\mbf{0},\ \mbf{e}_{n}} \right) \notag \\
& = \frac{\eta(1)}{|G|} \sum\limits_{g \in G} \eta(g^{-1}) \sum\limits_{\substack{t \in \mfk{C}\\ t(\mbf{0})= \mbf{e}_{n}}} \ups_{gtg^{-1}} \notag \\
& = \frac{\eta(1)}{|G|} \sum\limits_{g\in G} \eta(g^{-1}) \sum\limits_{\substack{t \in \mfk{C}\\ tg(\mbf{0})= g(\mbf{e}_{n})}}\ups_{t} \notag \\
& = \frac{\eta(1)}{|G|}\sum\limits_{t \in \mfk{C}} \left(\sum\limits_{\substack{g \in G\\ tg(\mbf{0})= g(\mbf{e}_{n})}} \eta(g^{-1}) \right) \ups_{t}.
\end{align}

Define $S:=\{g \in G\ :\ \mbf{c}g(\mbf{0})) = g(\mbf{e}_{n})\}$. Then the coefficient of $\ups_{\mbf{c}}$ in the above sum is $\sum\limits_{s \in S} \eta(s^{-1})$. Computing the rank of $\mfk{m}$ is now reduced to computing some character sums. In \S~\ref{sec:charactersums}, we compute these character sums later and obtain Lemma~\ref{lem:charactersums}.

Using Lemma~\ref{lem:charactersums}, \eqref{eq:imageiso}, and \eqref{eq:isotypiccomponent}, we showed that
(i) for $\eta \in \{\one,\ \theta,\ \alpha,\ \beta \}$, we have $\mfk{m}(\vp_{\eta}) \neq 0$; and (ii)$\mfk{m}(\vp_{\psi}) = 0$. Since $\mfk{m}$ is a $G$-module map, by Schur's lemma, for all 
$\eta \in \{\one,\ \theta,\ \alpha,\ \beta \}$, we have $\mfk{m}(\vp_{\eta}) \cong \vp_{\eta}$. So by \eqref{eq:rankcomprision} and \eqref{eq:imagem}, we have 

\[\mrm{Rank}(M(G))\geq \mrm{Rank}(\mfk{m}) = \sum\limits_{\eta \in \{\one,\ \theta,\ \alpha,\ \beta \}} \eta(1)= (2^{n}-1)(2^{n}-2).\]

Now, using \eqref{eq:rankup}, we have $\mrm{Rank}(M(G))=(2^{n}-1)(2^{n}-2)$. Theorem~\ref{thm:agln2sker} now follows from Theorem~\ref{thm:modulemethod}.
\section{Proof of Theorem~\ref{thm:agln2stablity}.}\label{sec:stability}
We start by establishing some notation. Let $G$ denote $\agl(n,2)$ and $\bb{C}^{G}$ be the group algebra of $\bb{C}$-valued functions on $G$. By $\Der(G)$, we denote the set of derangements in $G$ and let $\de_{G}:=|\Der(G)|$. 
Given $S\subset V$, let $\ind_{S} \in \bb{C}^{G}$ denote the indicator function of $S$ in $G$. Given $x,y \in V$, by $\ind_{x,y}$, denote the indicator function of $\I_{x,y}:=\{g \in G\: g(x)=y\}.$ Given $\eta \in \irr(G)$, let $U_{\eta}$ be the $\eta$-isotypic component in $\bb{C}^{G}$, and let $\lambda_{\eta}:= \dfrac{\eta(1)}{|G|} \sum\limits_{d \in \Der(G)} \eta(d)$. From Lemma~\ref{lem:spec}, $\{\lambda_{\eta}\ :\ \eta \in \irr(G)\}$ is the set of eigenvalues of the derangement graph $\Gamma_{G}$. We set $U_{G}:=U_{\one}+U_{G}$.

Let $\psi \in \irr(G)$ be such that $\one+\psi$ is the permutation character associated with the action of $G$ on $V$. Given $g \in G$, if $\mrm{fix}(g)$ denotes the number of points in $V$ fixed by $g$, then $\psi(g)=\mrm{fix}(g)-1$. It now follows that
\begin{equation}\label{eq:stdevalue}
\lambda_{\psi}= \dfrac{-\de_{G}}{2^{n}-1}\ \ \&\ \ \lambda_{\one}=\de_{G}.  
\end{equation}
We first show that $\lambda_{\psi}$ is the least eigenvalue of $G$ and that $U_{\psi}$ is the $\lambda_{\psi}$-eigenspace. We also find a bound on the size of the second-smallest eigenvalue. Let $A_{G}$ denote the adjacency matrix of the derangement graph $\Gamma_{G}$.
To help us find the least and second-smallest eigenvalues of $A_{G}$, we now find some bounds on eigenvalues. We know that $\mrm{Tr}(A^{2}_{G})=2|E(\Gamma_{G})|=\de_{G}|G|$. From Lemma~\ref{lem:spec}, we have 
\begin{align*}
\de_{G}|G| &= \mrm{Tr}(A^{2}_{G}) = \sum\limits_{\chi \in \mrm{Irr}(G)} \left(\chi(1)\right)^{2}\lambda^{2}_{\chi} 
\end{align*}
Setting $\pr_{G}={d_{G}}/{|G|}$, from the above equation, for all $\chi \in \mrm{Irr}(G)$, we have 
\begin{equation}\label{eq:eigenbound}
|\lambda_{\chi}| \leq \frac{\de_{G}}{\chi(1)\sqrt{\pr_{G}}}.
\end{equation}

 To obtain information on degrees of irreducible characters of $G$, we use the following result from \cite{tiep1996minimal} on the degrees of irreducible characters of $\gl(n,2)$.
\begin{theorem}(Theorem 1.1 of \cite{tiep1996minimal}):
Let $n \geq 5$ and let \[1=d_{0}<d_{1}<d_{2}<\cdots <d_{l}\] be the character degrees of $\gl(n,2)$, then
\begin{enumerate}
\item $d_{1}= 2^{n}-2$, $d_{2}=\begin{cases} \dfrac{(2^{n}-1)(2^{n-1}-4)}{3}\ \text{if $n\neq 6$}\\ 217\ \text{otherwise} \end{cases}$; and
\item there is a unique character of degree $d_{1}$.
\end{enumerate}
\end{theorem}

The group $K$ is isomorphic to $\agl(m-1,2)$. So Lemma~\ref{lem:affineirr} builds the irreducible characters of $G$, recursively.
Using induction and the above two results, we obtain the following result.

\begin{cor}\label{cor:chardegaffine}
Let $n\geq 5$ and $n\neq 6$. If $1=r_{0}<r_{1}<r_{2}\ldots<r_{l}$ are the character degrees of $G$, then
\begin{enumerate}
\item $r_{1}=2^{n}-2$ and there is a unique character of degree $r_{1}$;
\item $r_{2}=2^{n}-1$ and there is a unique character of degree $r_{2}$; and
\item $r_{3}= \begin{cases} \dfrac{(2^{n}-1)(2^{n-1}-4)}{3}\ \text{if $n\neq 6$}\\ 217\ \text{otherwise.} \end{cases}$
\end{enumerate}
\end{cor}

Assume that $n\geq 5$ and $n\neq 6$. From the above corollary, $\psi$ is the unique character of $G$ of degree $2^{n}-1$. Then by \eqref{eq:stdevalue}, we have 
\begin{equation}\label{eq:stdevalueaffine}
\lambda_{\psi}= -\de_{G}/(2^{n}-1).
\end{equation}

Let $\theta \in \irr(G)$ be the character of degree $2^{n}-2$ as defined in \S~\ref{sec:characters}.  We have $\theta(M,v)=|\mrm{Ker}(M-I)|-2$, for all $(M,v) \in G$. By Corollary~\ref{cor:chardegaffine}, $\theta$ is the unique character of degree $2^{n}-2$. We now consider $\lambda_{\theta}$. We observe that $(M,v) \in \Der(G)$ if and only if $v \notin \mrm{Im}(M-I)$. For all $M \in \gl(n,2)$, note that $\theta(M) \geq 0$ if and only if $\mrm{Im}(M-I)\neq V$. Therefore, we have $\theta(N,x)\geq 0$ for all $(N,x)\in \Der(G)$. We also note that $(I,v)\in \Der(G)$ and that $\theta((I,v))=2^{n}-2$. We have
\begin{equation}\label{eq:theta}
\lambda_{\theta}\gneqq 0.
\end{equation}

 By Corollary~\ref{cor:chardegaffine}, for $\chi \in \irr{G}\setminus\{\one,\ \theta,\ \psi\}$,  we have 
\[\chi(1) \geq \psi(1) \frac{2^{n-1}-4}{3}.\] Now using \eqref{eq:eigenbound} and \eqref{eq:stdevalueaffine}, for all $\chi \in \irr(G) \setminus \{\one,\ \theta,\psi \}$
\begin{equation}\label{eq:eigenboundaffine0}
|\lambda_{\chi}| \leq \frac{3|\lambda_{\psi}|}{(2^{n-1}-4)\sqrt{\pr_{G}}}.
\end{equation}
From Theorem 1.1 of \cite{spiga2017number}, for $n\geq 5$, we have 
\[\pr_{G}=\frac{|D(G)|}{|G|}= \sum\limits_{i=1}^{n} \frac{(-1)^{i-1}}{2^{i(i+1)/2}} \geq \frac{1}{2}-\frac{1}{8}=\frac{3}{8}.\]
Now, using \eqref{eq:eigenboundaffine0}, we have 
$\chi \in \irr(G) \setminus \{\one,\ \theta,\psi \}$
\begin{equation}\label{eq:eigenboundaffine}
|\lambda_{\chi}| \leq \mrm{min}\left(\frac{|\lambda_{\psi}|}{2^{n-6}},\ \frac{|\lambda_{\psi}|}{2}\right).
\end{equation}
The above result is true for $n\geq 5$ and $n\neq 6$. With the help of a computer, we can verify that it is true for all $n\geq 4$. 

Let $\lambda$ be the smallest eigenvalue of $\Gamma_{G}$. Using, Lemma~\ref{lem:spec}, \eqref{eq:stdevalueaffine}, \eqref{eq:eigenboundaffine}, and \eqref{eq:theta}, we see that $\lambda=-\frac{\de_{G}}{2^{n}-1}$ and that $\mrm{Ker}(A_{G}-\lambda I)=U_{\psi}$. Let $\mu$ be the second-smallest eigenvalue. If $\lambda_{\theta}$ also satisfies the inequality \eqref{eq:eigenboundaffine}, then $|\mu|$ is also bounded by the right-hand side of \eqref{eq:eigenboundaffine}. If $\lambda_{\theta} \gneqq \left(\frac{|\lambda_{\psi}|}{2^{n-6}},\ \frac{|\lambda_{\psi}|}{2}\right)$, then by \eqref{eq:eigenboundaffine0}, $\lambda_{\theta}$ cannot be the second-smallest eigenvalue. Therefore, we have 
\[|\mu|\leq \mrm{min}\left(\frac{|\lambda_{\psi}|}{2^{n-6}},\ \frac{|\lambda_{\psi}|}{2}\right).\] Now applying Theorem~\ref{thm:stabilityratiobound}, we have
for any intersecting set $S$ of size $c|\gl(n,2)|$, we have 
\begin{equation}\label{eq:distanceboundaffine}
||\ind_{S}- \mrm{Proj}_{U_{G}}(\ind_{S})||^{2} \leq (1+O(1/2^{n-6}))(1-c)\frac{c}{2^{n}-1},
\end{equation}
where $U_{G}:=U_{\one}+U_{\psi}$.

 Plaza (cf. Proposition~10 of \cite{plaza2015stability}) observed that the main result of \cite{ellis2015quasi} result extends to all $3$-transitive groups. Using Plaza's observation, we have the following characterization of Boolean functions in $\bb{C}^{G}$ which are ``close'' in Euclidean distance to the subspace $U_{G}$.

\begin{theorem}\label{thm:EFF}(Proposition 10 of \cite{plaza2015stability}) There exists absolute constants $C_{0}$ and $\epsilon_{0}\in [0, \frac{1}{2}]$ such that the following holds. Let $S \subset G$ with $|S|= \frac{c|G|}{2^{n}}$ with $c \in (\frac{1}{2}, 1]$. If \[||\ind_{S}-\Proj_{U_{G}}(\ind_{S})||^2 \leq \frac{\epsilon c}{2^{n}},\] where $\epsilon \leq \epsilon_{0}$, then there is a canonical intersecting set $T$ in $G$ such that
\[|S \Delta T| \leq C_{0}c^{2} \frac{|G|}{2^{n}}\left( \sqrt{\epsilon} + \frac{1}{2^{n}} \right).\]
\end{theorem}

As a corollary, we have the following result.

\begin{cor}\label{cor:affinesymdiff}
Given $\phi>0$, there exists $c(\phi) \in (1/2,\ 1)$ and such that the following is true. If $n\geq n(\phi)$ and $S \subset \agl(n,2)$ is an intersecting set of size $c|\gl(n,2)|$ with $c\geq c(\phi)$, then there is a canonical intersecting set $T$ such that 
\[|S \Delta T| \leq \phi c^{2} \frac{|\agl(n,2)|}{2^{n}-1}.\] 
\end{cor}
\begin{proof}
Let $C_{0}$ and $\epsilon_{0}$ be the constants in Theorem~\ref{thm:EFF}. We first pick $0<c_{1}<1$ such that 
\begin{equation}\label{eq:ccond}   
2(1-c_{1})\lneqq \mrm{min}(\{\frac{\phi^{2}}{C^{2}_{0}},\ \epsilon_{0}\}).\end{equation} 
Next, we pick $n(\phi)$ such that
\begin{equation} \label{eq:ncond}
\frac{1}{2^{n(\phi)}} \leq \frac{\phi}{C_{0}} -\sqrt{2(1-c_{1})}.\end{equation}

Let $n\geq n(\phi)$ and $S$ be an intersecting set in $G:=\agl(n,2)$ of size $c|\gl(n,2)|$, where $c\geq c_{1}$.
From \eqref{eq:distanceboundaffine} and \eqref{eq:ccond}, we have 
 \[ ||\ind_{S}- \mrm{Proj}_{U_{G}}||^{2} \leq  2(1-c)\frac{c}{2^{n}-1} \leq \frac{\epsilon_{0}c}{2^{n}-1}.\]
 
Application of  Theorem~\ref{thm:EFF} yields that there is a canonical intersecting set $T$ such that 
\[|S \Delta T| \leq \frac{c^{2}|G|}{2^{n}}C_{0}\left(\sqrt{2(1-c)}+ \frac{1}{2^{n}}\right) \leq \frac{c^{2}|G|}{2^{n}}C_{0}\left(\sqrt{2(1-c_{1})}+ \frac{1}{2^{n}}\right).\]
Since $n\geq n(\phi)$, using \eqref{eq:ncond}, we have \[S\Delta T \leq \phi c^{2} \frac{|G|}{2^{n}-1}.\]

Let $c_{2} \in (1/2, 1)$ be such that $c_{2}|\gl(n,2)|> |\gl(n,2)|-1$, for all $n<n(\phi)$. We set $c(\phi) := \mrm{max}(\{c_{1},c_{2}\})$. From the above, the result is true for all $n\geq n(\phi)$. In the case $n<n(\phi)$, the result follows from Theorem~\ref{thm:agln2sker}. 
\end{proof}
We are now ready to prove our main result
\paragraph{Proof of Theorem~\ref{thm:agln2stablity}:}  Choose $\phi=\frac{3}{8}$. Let $c(\phi)$ in the above result. Let $c \geq c(\phi)$ and let $S$ be an intersecting set in $\agl(n,2)$ of size $c|\agl(n,2)|/(2^{n}-1)$. By Corollary~\ref{cor:affinesymdiff}, there exists a canonical intersecting set $T$ with $|S\Delta T| \leq \frac{3}{8}|T|$. We claim that $S \subset T$. Assume the contrary. In the derangement graph $\Gamma_{G}$, by Corollary~\ref{cor:ratiobound}, any $s$ in $S\setminus T$ has exactly $\dfrac{\de_{G}}{2^{n}-1}$ neighbours in $T$. Thus, since $S$ is an independent set, $|T\setminus S| \geq \frac{\de_{G}}{2^{n}-1} $. Noting that $\de_{G}= \pr_{G}|G| \geq \frac{3|G|}{8}$, we have 
 \begin{align*}
 |S \Delta T| -|T \setminus S| & \leq \frac{3|G|}{8(2^{n})} - \frac{\de_{G}}{2^{n}-1}\\
 & \leq \frac{3|G|}{8(2^{n})} - \frac{3|G|}{8(2^{n}-1)} <0.
\end{align*}   
This is absurd and so our assumption that $S \not\subseteq T$ is false. Thus $S\subset T$ and this proves the result. \qed

\section{Proof of Lemma~\ref{lem:charactersums}}\label{sec:charactersums}

We first collect some auxiliary results on character sums. The following is a well known result (see for instance Lemma~4.1 of \cite{isaacs2008character}).

\begin{lemma}\label{lem:zerosum}
Let $L$ be a finite group and $P:L \to \gl_{n}(\bb{C})$ be an $n$-dimensional representation affording the character $\zeta$. If $\left\langle \zeta,\ \one \right\rangle=0$, then
$\sum\limits_{l \in L}P(l)=0$.
\end{lemma}

Using this, we now derive the following result for permutation representations. The proof is inspired by the proof of Corollary 4.2 of \cite{isaacs2008character}.

\begin{cor}\label{cor:charsums}
Let $M$ be a finite group acting on a set $\Omega$, let $\rho$ be the corresponding permutation character, and consider a subgroup $L \leq M$. If $\{O_{1},\ \ldots,\ O_{k}\}$ is the set of $L$-orbits on $\Omega$, then for all $x \in M$, we have 
\[\sum\limits_{y\in L} \rho(xy) = \left(\sum\limits_{i=1}^{k} \frac{|O_{i} \cap x(O_{i})|}{|O_{i}|} \right) |L|,\]
where $x(O_{i})=\{x\cdot o\ : \ o \in O_{i}\}$, for all $i \in [k]$.
\end{cor}
\begin{proof}
 We consider the vector $\bb{C}[\Omega]$ spanned by the formal basis $\{\ups_{\omega}\ :\ \omega \in \Omega\}$. The action of $G$ on $\Omega$ endows $\bb{C}[\Omega]$ with a $G$-module structure.
 Given $X\subseteq \Omega$, define $\ups_{X}:=\sum\limits_{x\in X} \ups_{x}$.
Let $U\leq \bb{C}[\Omega]$ be the subspace spanned by $\{\ups_{O_{i}}\ :\ i \in [k]\}$. We note that $U$ is a $\bb{C}H$-submodule of $\bb{C}[\Omega]$. We endow the space $\bb{C}[\Omega]$ with the standard inner product $\left\langle \cdot\ ,\ \cdot \right\rangle$ with $\{\ups_{\omega}\ :\ \omega \in \Omega\}$ as an orthonormal basis.
Noting that the standard dot product on $\bb{C}[\Omega]$ is $L$-invariant, we note that $U^{\perp}$ is also a $\bb{C}L$-module. Let $\beta$ be an ordered orthogonal basis of $\bb{C}[\Omega]$ with $\beta_{i}:=\ups_{O_{i}}$ for all $i \in [k]$, and $\beta_{i} \in U^{\perp}$ for all $i \geq k+1$. Given $g \in G$, let $P(g)$ be the matrix with respect to $\beta$ of the linear map defined by the action of $g$ on $\bb{C}[\Omega]$. The map $P : L \to \gl_{|\Omega|}(\bb{C})$ is a representation affording the character $\rho$.

As $U^{\perp}$ is an $L$-submodule and $L$ acts trivially on $U$, there is an $L$-representation\\ $Q: L \to \gl_{|\Omega|-k}(\bb{C})$ such that $P(l)= \left(\begin{array}{c | c}
I & 0 \\
\hline
0 & Q(l)
\end{array}  \right),$ for all $l \in L$. As the action of $L$ has $k$ orbits, by the orbit-counting lemma, we have $\left\langle \rho|_{L},\ \one \right\rangle_{L}=k$. Therefore, if $Q$ affords the character $\chi$, we must have $\left\langle \chi,\ \one \right\rangle_{L}=0$. Now using Lemma~\ref{lem:zerosum}, we have 

\begin{equation}\label{eq:lsum}
 \sum\limits_{l \in L}P(l) = \left(\begin{array}{c | c}
|H|I & 0 \\
\hline
0 & 0
\end{array}  \right).
\end{equation} 

 By definition of $P$, given $i,j \in[k]$, and $x \in G$, we have 
\begin{equation}\label{eq:ij}\left(P(x)\right)_{i,j}= \frac{\left\langle \ups_{O_{i}}, \ups_{x(O_{j})} \right \rangle}{\left\langle\ups_{O_{i}},\ups_{O_{i}}\right\rangle}= \frac{|O_{i}\cap x(O_{j})|}{|O_{i}|}.
\end{equation}

Since \[\sum\limits_{l \in L}\rho(xh)  = \mrm{Trace}\left(P(x)\left[\sum\limits_{l \in L}P(l) \right]\right),\] 
the result now follows by using \eqref{eq:lsum} and \eqref{eq:ij}.
 
\end{proof}

We recall the set $S=\{g \in G\ :\ \mbf{c}g(\mbf{0})) = g(\mbf{e}_{n})\}$, where $\mbf{c}:=(J_{n},\ \mbf{e}_{n})$. We note that  both the centralizer $\mbf{C}$ of $\mbf{c}$ and the pointwise stabilizer $H$ of $(\mbf{0},\ \mbf{e}_{n})$ are subsets of $S$. We will now show that $S=H\mbf{C}$. We do this in a series of lemmas. 

\begin{lemma}\label{lem:centralizer}
$\mbf{C}$ is a regular subgroup of $G$.
\end{lemma} 
\begin{proof}
We recall that $\mbf{c}=(J_{n},\ \mbf{e}_{n})$.
Given $(N,v)\in \agl(n,2)$, we have 
\[(N,v)(J_{n},\ \mbf{e}_{n})(N,v)^{-1}=(NJ_{n}N^{-1},\ v+N\mbf{e}_{n} + NJ_{n}N^{-1}v).\]

Therefore, $(N,v) \in \mbf{C}$ if and only if the following are true:
(i)$N \in \mbf{C}_{\gl(n,2)}(J_{n})$; and (ii) $(J_{n}+I)v= (N+I)\mbf{e}_{n}$. Given $\overrightarrow{a} \in \bb{F}^{n-1}_{2}$, define the matrix $N^{\overrightarrow{a}}$ to be the upper triangular matrix with $N^{\overrightarrow{a}}_{ij}=a_{n-(j-i)}$ and $N^{\overrightarrow{a}}_{ii}=1$, for all pairs $j,i \in [n]$ with $j>i$. 
We recall from linear algebra that $\mbf{C}_{\gl(n,2)}(J_{n}) =\{N^{\overrightarrow{a}}:\ \overrightarrow{a} \in \bb{F}^{n-1}_{2}\}$. We observe that the equation $(J_{n}+I)v=(N^{\overrightarrow{a}}+I)\mbf{e}_{n}$ has two solutions, namely $x_{\overrightarrow{a}}:=\begin{pmatrix}
0  \\   \overrightarrow{a}
\end{pmatrix}$ and $ y_{\overrightarrow{a}} :=\begin{pmatrix}
1 \\ \overrightarrow{a}
\end{pmatrix}$. We can now conclude that 
\begin{equation}\label{eq:centralizer}
\mbf{C}:=\{(N^{\overrightarrow{a}},\ x_{\overrightarrow{a}}),\ (N^{\overrightarrow{a}},\ y_{\overrightarrow{a}}) :\ \overrightarrow{a} \in \bb{F}^{n-1}_{2} \}.
\end{equation}

Note that when $\overrightarrow{a} \neq \mbf{0}$, the matrix $N^{\overrightarrow{a}}+I$ is in row-reduced echelon form. The augmented matrices $\left[N^{\overrightarrow{a}}+I\ |\ x_{a}\right]$ and $\left[N^{\overrightarrow{a}}+I\ |\ y_{a}\right]$ are also in row-reduced echelon forms. Counting the number of pivots, we see that when $\overrightarrow{a} \neq \mbf{0}$, we have \[\mrm{rank}\left(\left[N^{\overrightarrow{a}}+I\ |\ x_{a}\right]\right)=\mrm{rank}\left(\left[N^{\overrightarrow{a}}+I\ |\ y_{a}\right]\right) =\mrm{rank}(N^{\overrightarrow{a}}+I) +1 .\] Therefore, when $\overrightarrow{a}\neq \mbf{0}$, both $x_{\overrightarrow{a}}$ and $y_{\overrightarrow{a}}$ are not in $\mrm{Im}(N^{\overrightarrow{a}})$. Therefore, every non-identity element of $\mbf{C}_{G}(G)$ is a derangement. Using $|\mbf{C}|=2^{n}$, we can conclude that $\mbf{C}$ is a regular subgroup.
\end{proof}

Let $K:=G_{\mbf{0}}$, the point stabilizer in $G$ of $\mbf{0}$. Since $\mbf{C}$ is a regular subgroup, we have $\mbf{C}\cap K=\{1_{G}\}$ and therefore $\mbf{C}K=G$. So given $s\in S$, there exists $x \in K$ and $y \in\mbf{C}$ such that $yx=s$. By definition of $S$, we have
\[\mbf{e}_{n}= s^{-1}\mbf{c}s(\mbf{0})= x^{-1}\mbf{c}x(\mbf{0})= x^{-1}\mbf{c}(\mbf{0})= x^{-1}(\mbf{e}_{n}).\] Therefore $x \in G_{(\mbf{0},\ \mbf{e}_{n})}=H$, and thus we have 
$S \subseteq \mbf{C}H$. By definition of $S$, we have $\mbf{C}H \subseteq S$. Thus, we have $S=\mbf{C}H$. Computing character sums of cosets of $H$ is beneficial in proving our result. We first start by proving part (2) of the Lemma~\ref{lem:charactersums}.

\subsection{The character $\psi$.}
The action of $G$ on $V$ affords $\pi:=\one+\psi$ as its permutation character. The action of $H$ on $V$ has three orbits: $O_{1}:=\{\mbf{0}\}$, $O_{2}:=\{\mbf{e}_{n}\}$, and $O_{3}:=V\setminus \{\mbf{0},\ \mbf{e}_{n}\}$. By Corollary~\ref{cor:charsums}, we have 
\[\sum\limits_{x \in \mbf{C}}\pi(xH)=\left(\sum\limits_{x\in \mbf{C}} |O_{1}\cap xO_{1}|+ |O_{2}\cap xO_{2}| + \frac{|O_{3}\cap xO_{3}|}{2^{n}-2} \right)|H|.\]
By Lemma~\ref{lem:centralizer}, as $\mbf{C}$ acts regularly on $V$, we note that $|O_{1}\cap xO_{1}|+ |O_{2}\cap xO_{2}|=0$, for every non-identity $x \in \mbf{C}$. By the description (see \eqref{eq:centralizer}) of elements of $\mbf{C}$, we have 
\begin{equation}\label{eq:image0em}
|\{\mbf{0},\ \mbf{e}_{n}\} \cap \{x(\mbf{0}),\ x(\mbf{e}_{n})\} | = \begin{cases}
0\ \text{if $x \in \mbf{C} \setminus \{Id,\ \mbf{c},\ \mbf{c}^{-1}\}$}\\
1\ \text{if $x \in \{\mbf{c},\ \mbf{c}^{-1}\}$, and}\\
2\ \text{if $x=Id$}.
\end{cases}
\end{equation}

Therefore, we have 
\[\sum\limits_{x \in \mbf{C}}\pi(xH) = \left(3 + 2\frac{(2^{n}-3)}{2^{n}-2} +(2^{n}-3)\frac{(2^{n}-4)}{2^{n}-2}\right)|H| = |\mbf{C}||H|.\]

As $\psi=\pi-\one$, we have
\begin{equation}\label{eq:psicharsum}
\sum\limits_{s \in \mbf{C}H}\psi(s^{-1})=0.
\end{equation}

So the coefficient of $\mbf{e}_{\mbf{c}}$ in the right-hand side of $\eqref{eq:imageiso}$ is $0$. For $t \in \mfk{C}$, with $t=xcx^{-1}$, the coefficient of $\ups_{t}$ is
$\dfrac{\psi(1)\psi(x\mbf{C}H)}{|G|}$. We show that this too is zero. Given $s\in \mbf{C}H$, there is a $z \in \mbf{C}$ such that $\mbf{C}Hs^{-1}=\mbf{C}Hz^{-1}$. As $\psi$ is a class function, we have
 \[\psi(\mbf{C}Hs^{-1})=\psi(\mbf{C}Hz^{-1})= \psi(z^{-1}\mbf{C}Hz^{-1}z)=\psi(\mbf{C}H)=0.\] By Lemma 2.1 of \cite{d2023cameron}, we can now conclude that $\psi(x\mbf{C}H)=0$. Therefore, the coefficient of all $\ups_{t}$ in the right-hand side of $\eqref{eq:imageiso}$ is zero, for all $t$. This shows part (2) of our lemma. Next, we prove the result for the character $\theta$. 
 \subsection{The character $\theta$.}
 Consider the $2$-transitive action of $\gl(n,2)$ on $V\setminus \{\mbf{0}\}$. This extends to a $2$-transitive action of $G$ on $V\setminus\{\mbf{0}\}$, with $V$ as its kernel. By the definition of $\theta$, the permutation character corresponding to this action is $\rho:=\one + \theta$. The action of $H$ on $V\setminus \{\mbf{0}\}$ has two orbits: $O_{1}=\{\mbf{e}_{n}\}$ and $O_{2}:=V\setminus \{\mbf{0},\ \mbf{e}_{n}\}$. Using Corollary \ref{cor:charsums} and $\theta=\rho-\one$, we compute that 
 \begin{equation}\label{eq:thetacharsum}
\sum\limits_{s \in \mbf{C}H}\psi(s^{-1})=|H|.
\end{equation}
\subsection{The character $\alpha$}
We consider the action of $G$ on $V^{\{2\}}$ (set of $2$-subsets of $V$) and the corresponding permutation character $\pi^{\{2\}}$. Using \eqref{eq:alphadef}, \eqref{eq:psicharsum}, and \eqref{eq:thetacharsum}, we have 
\begin{equation}\label{eq:alphapirelation}
\alpha(\mbf{C}H)= \pi^{\{2\}}(\mbf{C}H)- (|C_{G}|+1)|H|.
\end{equation}

As discussed after \eqref{eq:Korbitssets}, the $H$-orbits on $V^{\{2\}}$ are:
\begin{itemize}
\item $O_{1}:=\{\{\mbf{0},\ \mbf{e}_{n}\}\}$,
\item $O_{2}:=\{\{\mbf{0},\ v\} :\ v \in V\setminus \{\mbf{0},\ \mbf{e}_{n}\}\}$,
\item  $O_{3}= \{\{\mbf{e}_{n},\ v\} :\ v \in V\setminus \{\mbf{0},\ \mbf{e}_{n}\}\}$, \item $O_{4}:=\{\{v,\ w\} :\ v, w \in V\setminus \{\mbf{0},\ \mbf{e}_{n}\}\ \&\ v+w=\mbf{e}_{n}\}$, 
\item and $O_{5}:=\{\{v,\ w\} :\ v,\ w,\ v+w \in V\setminus \{\mbf{0},\ \mbf{e}_{n}\}\}$.
\end{itemize}

To apply Corollary~\ref{cor:charsums}, we now compute $|O_{i} \cap x(O_{i}) |$ for all $x \in \mbf{C}$ and for all $i \in [5]$. From the description of $\mbf{C}$ in \eqref{eq:centralizer}, we note that given $x\in C$, there is a unique $\ovr{a}_{x} \in \bb{F}^{n-1}_{2}$ and a unique $b_{x} \in \left\{ \left( \begin{array}{c}
0 \\ \ovr{a}
\end{array} \right),\ \left(\begin{array}{c}
1 \\ \ovr{a}
\end{array}  \right\}\right)$, such that $x=(N_{\ovr{a}_{x}},\ b_{x})$. Let $v_{0,x}=x(\mbf{0})$, $v_{n,x}=x(\mbf{e}_{n})$, and $z_{x}:=N_{\ovr{a}}(\mbf{e}_{n})$. We note that, given $v,w\in V$, we have $x\cdot v+x \cdot w =N_{\ovr{a}_{x}}(v+w)$. 
From this, we observe the following:
\begin{subequations}\label{eq:imagesoforbits}
\begin{align}
xO_{1} & =\{\{v_{0,x},\ v_{n,x}\}\}\\
xO_{2} &= \{\{v_{0,\ x},\ w \}\ :\ w\in V \setminus \{v_{0,x},\ v_{n,\ x}\}\} \label{eq:xO2}\\
xO_{3} &= \{\{v_{n,\ x},\ w \}\ :\ w\in V \setminus \{v_{0,x},\ v_{n,\ x}\}\} \label{eq:xO3}\\
xO_{4} &= \{ \{v,\ w\}\ :\ v,w \in V \setminus \{v_{0,x},\ v_{n,\ x}\}\ \&\ v+w=z_{x} \}\\
xO_{5} &= \{ \{v,\ w\}\ :\ v,w \in V \setminus \{v_{0,x},\ v_{n,\ x}\}\ \&\ v+w \notin \{z_{x},\ \mbf{0}\} \}.
\end{align}
\end{subequations}

We now compute $|O_{i}\cap xO_{i}|$. We start with some auxiliary computations. Given $x\in \mbf{C}$ and $z\in V\setminus \{\mbf{0}\}$, we define, 
\begin{align*}
L_{x,z} :=&\{\{v,\ w\}\ :\ v,w \in V \setminus \{v_{0,x},\ v_{n,\ x}\}\ \&\ v+w=z \},\ \text{and}\\
A_{x,z} :=& \{\{v_{0,x},\ v_{0,x}+z\},\ \{v_{n,x},\ v_{n,x}+z\}   \}.
\end{align*}
 For ease of notation, set 
$L_{z}:=L_{Id,z}$. As $x\cdot v+ x \cdot w=N_{\ovr{a}_{x}}(z)$, for all $x\in \mbf{C}$ and $z\in V\setminus \{\mbf{0}\}$, we have 
\begin{subequations}\label{eq:linerelations}
\begin{align}
|L_{z}| &=  (2^{n-1}-2 + \delta_{z,\mbf{e}_{n}}); \label{eq:linesize} \\
L_{x,\ N_{\ovr{a}_{x}}z} &= x L_{z};\label{eq:lineim}\\
L_{z} \cap L_{x,z} &= L_{z}\setminus A_{x,z}; \label{eq:lineint}\\
xO_{4} & = L_{x,\ z_{x}}; \ \text{and}\\
xO_{5} &= \bigcup\limits_{z\in V \setminus \{\mbf{0},\ z_{x}\}} L_{x,z}. \label{eq:O5disunion}
\end{align}
\end{subequations}

We have $z_{x}=\mbf{e}_{n}$ if and only if $x=Id$ or $x=\mbf{s}:=\{Id,\ \mbf{e}_{1}\}$. By the definitions of $L_{x,z}$, for all $x\in \mbf{C}$ and $z\neq \tilde{z}$, we have $L_{z}\cap L_{x, \tilde{z}}=\emptyset$. By using the above relations, we have 
\begin{equation}\label{eq:O4int}
|O_{4}\cap xO_{4}|= \begin{cases} 2^{n-1}-1 & \text{if $x=Id$},\\
2^{n-1}-2 & \text{if $x=\mbf{s}$},\\
0& \text{otherwise.}
\end{cases}
\end{equation}

Now consider $|O_{5} \cap xO_{5}|$. Using \eqref{eq:O5disunion}, we first note that 
\[O_{5}\cap xO_{5}= \bigcup\limits_{z\in V\setminus \{\mbf{0},\ \mbf{e}_{n},\ z_{x}\}} L_{x}\cap L_{x,z}.\] 

Using \eqref{eq:lineint} and \eqref{eq:linesize}, we see that
\begin{equation}\label{eq:O5int-0}
|O_{5}\cap xO_{5}| = \sum\limits_{z \in V\setminus \{\mbf{0},\ \mbf{e}_{n},\ z_{x}\}} \left(2^{n-1}-2-|L_{z}\cap A_{x,z}|\right).
\end{equation}

First, we note that since $v_{0,x}+v_{n,x}=N_{\ov{a}_{x}}(\mbf{e}_{n})=z_{x}$, for all $z\neq z_{x}$, $v_{0,x}+z\neq v_{n,x}$. Therefore, for all $z\neq z_{x}$, we have $\{v_{0,x},\ v_{0,x}+z\}\cap \{v_{n,x},\ v_{n,x}+z\}=\emptyset$ and thus $|A_{x,z}|=2$. Since $\{a,b\} \in L_{z}$ if and only if $\{a,b\}\cap\{\mbf{0},\ \mbf{e}_{n}\}=\emptyset$, we note that 
\begin{equation}\label{eq:LAint}
|L_{z} \cap A_{x,z}|= 2-|\{v_{n,x},\ v_{n,x}+z,\ v_{0,x},\ v_{0,x}+z\}\cap \{\mbf{0}, \mbf{e}_{n}\}|.
\end{equation}

We now compute $|O_{5}\cap xO_{5}|$. This shall be accomplished in the following  cases:

\paragraph{Case 1:} Assume that $x\notin \{Id,\ \mbf{c},\ \mbf{c}^{-1},\ \mbf{s}\}$. In this case, we have \[\{z_{x},\ v_{0,x},\ v_{n,x}\} \cap \{\mbf{0},\ \mbf{e}_{n}\}= \emptyset.\] 

So by \eqref{eq:LAint}, for all $z\notin \{\mbf{0},\ \mbf{e}_{n},\ \mbf{z}_{x}\}$, we have
\[|L_{z} \cap A_{x,z}|= 2-|\{\ v_{n,x}+z,\ v_{0,x}+z\}\cap \{\mbf{0}, \mbf{e}_{n}\}|.\]

Since $v_{0,x}+v_{n,x}=z_{x}$, we have 
\[ \{ v_{n,x}, v_{0,x},v_{n,x}+\mbf{e}_{n}, v_{0,x}+\mbf{e}_{n}\} \cap \{\mbf{0},\mbf{e}_{n}, z_{x}\} =\emptyset.\]

We therefore have for $z \notin \{\mbf{0},\ \mbf{e}_{n},\ z_{x}\}$,
\begin{equation}\label{eq:laintergenx}
|L_{z} \cap A_{x,z}| = \begin{cases}
                        2  &
    \text{if $z \notin \{ v_{n,x}, v_{0,x},v_{n,x}+\mbf{e}_{n}, v_{0,x}+\mbf{e}_{n}\}$ } \\
                        0  & \text{otherwise.}
                        \end{cases} 
\end{equation}

Now using \eqref{eq:O5int-0}, we see that for all  $x\notin \{Id,\ \mbf{c},\ \mbf{c}^{-1},\ \mbf{s}\}$, we have 
\begin{equation}\label{eq:O5intgenx}
|O_{5} \cap xO_{5}|= (2^{n}-7)(2^{n-1}-4)+4(2^{n-1}-3)=2^{2n-1}-11\times(2^{n-1})+16.
\end{equation}

\paragraph{Case 2:} Assume that $x=\mbf{c}$. In this case, we have $v_{0,x}=\mbf{e}_{n}$, $v_{n,x}=\mbf{e}_{n-1}$, and $z_{x}=\mbf{e}_{n}+\mbf{e}_{n-1}$.
So by \eqref{eq:LAint}, for all $z\notin \{\mbf{0},\ \mbf{e}_{n},\ \mbf{z}_{x}\}$, we have
\[|L_{z} \cap A_{x,z}|= 1-|\{\ \mbf{e}_{n-1}+z,\ \mbf{e}_{n}+z\}\cap \{\mbf{0}, \mbf{e}_{n}\}|.\]

We therefore have for $z \notin \{\mbf{0},\ \mbf{e}_{n},\ z_{x}\}$,

\begin{equation}\label{eq:lainterc}
|L_{z} \cap A_{x,z}| = \begin{cases}
                        0 &
    \text{if $z = \mbf{e}_{n-1}$ } \\
                        1  & \text{otherwise.}
                        \end{cases} 
\end{equation}
Now using \eqref{eq:O5int-0}, we have

\begin{equation}\label{eq:O5intc}
|O_{5} \cap \mbf{c}O_{5}|= (2^{n}-4)(2^{n-1}-3)+(2^{n-1}-2)=2^{2n-1}-9\times(2^{n-1})+10.
\end{equation}

As $|O_{5} \cap \mbf{c}O_{5}|=|\mbf{c^{-1}O_{5}} \cap O_{5}|$, we also have

\begin{equation}\label{eq:O5intcinv}
|O_{5} \cap \mbf{c}O_{5}|= (2^{n}-4)(2^{n-1}-3)+(2^{n-1}-2)=2^{2n-1}-9\times(2^{n-1})+10.
\end{equation}

\paragraph{Case 3:} Assume $x=\mbf{s}$. We have $z_{x}=\mbf{e}_{n}$, $v_{0,x}=\mbf{e}_{1}$, and $v_{n,x}=\mbf{e}_{1}+\mbf{e}_{n}$. So by \eqref{eq:LAint}, for all $z\notin \{\mbf{0},\ \mbf{e}_{n}\}$, we have
\[|L_{z} \cap A_{x,z}|= 2-|\{\ \mbf{e}_{1}+z,\ \mbf{e}_{1}+\mbf{e}_{n}+z\}\cap \{\mbf{0}, \mbf{e}_{n}\}|.\]

We therefore have for $z \notin \{\mbf{0},\ \mbf{e}_{n} \}$,

\begin{equation}\label{eq:lainters}
|L_{z} \cap A_{x,z}| = \begin{cases}
                        2 &
    \text{if $z \notin \{\mbf{e}_{1},\ \mbf{e}_{1}+\mbf{e}_{n}\}$, } \\
                        0  & \text{otherwise.}
                        \end{cases} 
\end{equation}

Now using \eqref{eq:O5int-0}, we have

\begin{equation}\label{eq:O5ints}
|O_{5} \cap \mbf{s}O_{5}|= (2^{n}-4)(2^{n-1}-4)+2(2^{n-1}-2)=2^{2n-1}-10\times 2^{n-1}+12.
\end{equation}

Now by \eqref{eq:O5intgenx}, \eqref{eq:O5intc}, \eqref{eq:O5intcinv}, and \eqref{eq:O5ints}, we have

\begin{equation}\label{eq:O5int}
|O_{5}\cap xO_{5}|= \begin{cases} 2^{2n-1}-6\times 2^{n-1}+4 & \text{if $x=Id$},\\
2^{2n-1}-9\times(2^{n-1})+10 & \text{if $x \in \{\mbf{c},\ \mbf{c}^{-1}\}$},\\
2^{2n-1}-10\times 2^{n-1}+12& \text{if $x=\mbf{s}$, and}\\
2^{2n-1}-11\times(2^{n-1})+16 & \text{otherwise.}
\end{cases}
\end{equation}

It follows from \eqref{eq:xO2} and \eqref{eq:xO3} that 

\begin{equation}\label{eq:O23int}
|O_{2}\cap xO_{2}|=|O_{3}\cap xO_{3}|= \begin{cases} 2^{n}-2 & \text{if $x=Id$},\\
0 & \text{if $x \in \{\mbf{c},\ \mbf{c}^{-1}\}$},\\
1 & \text{otherwise.}
\end{cases}
\end{equation}

It is also clear that 
\begin{equation}\label{eq:O1int}
|O_{1}\cap xO_{1}|= \begin{cases}
    1 &\text{if $x=Id$ and} \\
    0 & \text{otherwise.}
\end{cases}.
\end{equation}

Now using Corollary~\ref{cor:charsums} along with \eqref{eq:O4int}, \eqref{eq:O1int}, \eqref{eq:O23int}, and \eqref{eq:O5int}, we compute that
\[\pi^{\{2\}}(\mbf{C}H)=(|C|+2)|H|,\] and thus, by \eqref{eq:alphapirelation}, we have 
\begin{equation}\label{eq:alphacharsum}
\sum\limits_{s \in \mbf{C}H}\alpha(s^{-1})=|H|.
\end{equation}
\subsection{The character $\beta$.}
We consider the action of $G$ on $V^{(2)}$ (set of pairs of distinct points in $V$) and the corresponding permutation character $\pi^{(2)}$. Using \eqref{eq:betadef}, \eqref{eq:psicharsum}, \eqref{eq:thetacharsum}, and \eqref{eq:alphacharsum}  we have 
\begin{equation}\label{eq:betapirelation}
\beta(\mbf{C}H)= \pi^{(2)}(\mbf{C}H)- (|\mbf{C}|+2)|H|.
\end{equation}

The $H$-orbits on $V^{(2)}$ are described in \eqref{eq:Horbitspairs}. To apply Corollary~\ref{cor:charsums}, we now compute $|\mfk{O}^{(2)}_{i} \cap x(\mfk{O}^{(2)}_{i}) |$ for all $x \in \mbf{C}$ and for all $i \in [8]$.

Since $\mbf{C}$ is a regular subgroup, for any non-identity $x\in \mbf{C}$, we have $x(\mbf{0})\neq \mbf{0}$ and $x(\mbf{e}_{n})\neq \mbf{e}_{n}$. Therefore, we have 

\begin{equation}\label{eq:pairorbitint1-6}
|\mfk{O}^{(2)}_{i} \cap x\mfk{O}^{(2)}_{i}|=\delta_{x,\ Id}\times |\mfk{O}^{(2)}_{i}|,\ \text{for all $1\leq i\leq 6$}.
\end{equation}   

We note that for $i=7,8$, we have $(v,w) \in \mfk{O}^{(2)}_{i}$ if and only if $(w,v) \in \mfk{O}^{(2)}_{i}$. Therefore, we have $|\mfk{O}^{(2)}_{7} \cap x \mfk{O}^{(2)}_{7}|=2|O_{4}\cap xO_{4}|$ and $|\mfk{O}^{(2)}_{8} \cap x \mfk{O}^{(2)}_{8}|=2|O_{5}\cap xO_{5}|$, where $O_{4}$ and $O_{5}$ are as defined in the previous subsection. Using Corollary~\ref{cor:charsums}, \eqref{eq:pairorbitint1-6}, \eqref{eq:O4int}, and \eqref{eq:O5int}, we compute that
\[\pi^{\{(2)\}}(\mbf{C}H)=(|C|+3+\frac{1}{2^{n-1}-1})|H|,\] and thus, by \eqref{eq:betapirelation}, we have 
\begin{equation}\label{eq:betacharsum}
\sum\limits_{s \in \mbf{C}H}\beta(s^{-1})=|H|\left(1+ \dfrac{1}{2^{n-1}-1}\right).
\end{equation} 
This concludes the proof of Lemma~\ref{lem:charactersums}.
\section{Further Work.}
We end with an open problem for future research. Theorem~\ref{cor:3t} gives a complete characterization of intersecting sets in the class of $3$-transitive groups. Theorem~\ref{thm:snstablity} due to Ellis, is a  Hilton--Milner type characterization of ``large'' intersecting sets in $\Sym(n)$. Plaza \cite{plaza2015stability} proved a similar result for intersecting sets in $\pgl(2,q)$, and our result, Theorem~\ref{thm:agln2stablity} is a  Hilton--Milner analogue for $\agl(n,2)$. We conjecture that intersecting sets in all $3$-transitive groups satisfy a Hilton--Milner type characterization.
\begin{conj}
There exists an absolute constant $c_{0}<1$ such that the following is true. If $G$ is a $3$-transitive group  of degree $n$ and $S$ is an intersecting set in $G$, with $|S|\geq c_{0}|G|/n$ , then $S$ is a subset of a canonical intersecting set in $G$.
\end{conj}
  
\bibliographystyle{plain}
\bibliography{ref}            
\end{document}